\documentclass[11pt]{amsart}
\usepackage{amsmath, amscd, amssymb}
\usepackage{latexsym, amssymb, amsmath, amscd, amsfonts}
\usepackage[mathscr]{eucal}
\usepackage{mathrsfs}
\usepackage{color}
\usepackage{concrete,euler,textcomp}
\usepackage{youngtab}

\numberwithin{equation}{section}
\newtheorem{theorem}{Theorem}[section]
\newtheorem{corollary}[theorem]{Corollary}
\newtheorem{definition}[theorem]{Definition}
\newtheorem{example}[theorem]{Example}

\newtheorem{lemma}[theorem]{Lemma}
\newtheorem{proposition}[theorem]{Proposition}
\newtheorem{remark}[theorem]{Remark}
\input xy
\xyoption{all}
\newcommand{\B}{\mathcal{B}}
\newcommand{\A}{\mathcal{A}}

\begin{document}

\title[Branching Multiplicity Spaces]{A Basis for the Symplectic Group Branching Algebra}
\author{Sangjib Kim}
\email{skim@maths.uq.edu.au}
\address{School of Mathematics and Physics, The University of Queensland, Brisbane, QLD 4072, Australia}
\author{Oded Yacobi}
\email{oyacobi@math.toronto.edu}
\address{ Department of Mathematics, University of Toronto, 40 St. George Street, Toronto, Ontario, Canada, M5S-2E4}


\begin{abstract}
The symplectic group branching algebra, $\B$, is a graded algebra whose
components encode the multiplicities of irreducible representations
of $Sp_{2n-2}(\mathbb{C})$ in each finite-dimensional irreducible representation of
$Sp_{2n}(\mathbb{C})$.  By describing on $\B$ an ASL structure, we
construct an explicit standard monomial basis of $\B$ consisting of
$Sp_{2n-2}(\mathbb{C})$ highest weight vectors.  Moreover, $\B$ is known to carry
a canonical action of the $n$-fold product $SL_{2} \times \cdots \times
SL_{2}$, and we show that the standard monomial basis is the unique
(up to scalar) weight basis associated to this representation.
Finally, using the theory of Hibi algebras we describe a deformation
of $Spec(\B)$ into an explicitly described toric variety.
\end{abstract}

\maketitle

\section{Introduction}

Let us consider a pair of complex reductive algebraic groups $G$ and $H$
with embedding $H\subset G$, and irreducible representations $V_{G}$
and $V_{H}$ of $G$ and $H$, respectively. A description of the
multiplicity of $V_{H}$ in $V_{G}$ regarded as a representation of
$H$ by restriction is called a \textit{branching rule} for $(G,H)$.
By Schur's lemma, the \textit{(branching) multiplicity space}, $
Hom_{{H}}(V_{H},V_{G}),$ encodes the branching rule.

In this paper, we study the branching multiplicity spaces for the symplectic
group ${Sp}_{2n}(\mathbb{C})$ of rank $n$ down to the symplectic group 
${Sp}_{2n-2}(\mathbb{C})$\ of rank $n-1$ by investigating the associated
\textit{branching algebra}. The branching algebra is a
subalgebra of the ring of regular functions over
${Sp}_{2n}(\mathbb{C})$.  Moreover, it is graded, having the branching multiplicity spaces as
its graded components:
\begin{equation*}
\B \cong \bigoplus_{\left( D,F\right) \in \Lambda _{n-1,n}}%
Hom_{{Sp}_{n-1}}(\tau _{n-1}^{D},\tau _{n}^{F}).
\end{equation*}%
Here $D$ and $F$ run over highest weights for $Sp_{2n-2}(\mathbb{C})$
and $Sp_{2n}(\mathbb{C})$ respectively, and $\tau_{2n-2}^{D}$ and
$\tau_{n}^{F}$ are the associated finite-dimensional irreducible representations.

\medskip

Branching rules for $(Sp_{2n}(\mathbb{C}), Sp_{2n-2}(\mathbb{C}))$, especially their combinatorial aspects, 
are well known (e.g., \cite{Le71, Ki71, Ki75, KT87, KT90, Pr94}). The main goal of this paper is to investigate the branching algebra $\B$ which governs the branching of symplectic groups.  Our main results are several-fold.  Firstly, we describe on $\B$ an algebra with straightening law (ASL) structure which presents $\B$ simply in terms of generators and relations.  Secondly, we show that this ASL structure is compatible with a canonical weight basis of $\B$ coming from a ``hidden symmetry'' of $\B$, namely an action of the n-fold product of $SL_2$ which acts irreducibly on the graded components of $\B$.  Finally, we unify these algebraic results with previously known combinatorial rules governing the branching of symplectic groups via a toric degeneration of B into an explicitly described toric variety.

\medskip

In Section 2, we construct the branching algebra $\B$ and review branching rules 
for $(Sp_{2n}(\mathbb{C}), Sp_{2n-2}(\mathbb{C}))$. Then in Section 3, we study 
$\B$ from the perspective of an ASL over a distributive lattice. Our first result shows that $\B$
has a natural standard monomial basis which satisfies simple straightening relations.
As a corollary we obtain a finite presentation of $\B$ in terms of generators
and relations.

Then, in Section 4, we recall a theorem from \cite{Yac09} which
shows that the natural $SL_{2}$ action on $\B$ can be canonically
extended to an action of
$$
L=SL_{2} \times \cdots \times SL_{2} \text{ (}n \text{ copies)}
$$
in such a way that each multiplicity space
$Hom_{{Sp_{2n-2}(\mathbb{C})}}(\tau_{n-1}^{D},\tau _{n}^{F})$ is an irreducible
$L$-module.  In particular, this theorem describes a canonical
decomposition of $\B$ into one-dimensional spaces.
Our second result shows this decomposition is compatible with the ASL structure on $\B$.
In other words, the standard monomial basis is the unique (up to scalar)
weight basis for the action of $L$ on $\B$.

In Section 5, we show that $\B$ can be flatly deformed into a
Hibi algebra, and, as a corollary, that $Spec(\B)$ is a deformation
of an explicitly described toric variety.  In particular, this connects our
enumeration of standard monomials with the more common description
of branching rules using diagrams of interlacing weights.

In future work we will apply these results to study properties of
the canonical weight basis for irreducible representations of the
symplectic group arising from this work.

\medskip

\noindent {\bf Acknowledgment.}  We express our sincere thanks to Roger Howe for drawing our attention to this project.  We also thank the two anonymous referees for their careful reading and  many helpful comments and suggestions which greatly improved the paper.

\medskip

\section{Branching algebra for $\left( {Sp}_{2n}(\mathbb{C}),{Sp}_{2n-2}(\mathbb{C})\right) $}
In this section we introduce our main object of study, the branching algebra for the pair
$( Sp_{2n}(\mathbb{C}), Sp_{2n-2}(\mathbb{C}))$.

\subsection{Symplectic Groups}

Let $Q_{n}=(q_{a,b})$ be the $n\times n$ matrix with $q_{a,n+1-a}=1$ for $%
1\leq a\leq n$ and $0$ otherwise. Then we define the symplectic group
$${G}_{n}=Sp(\mathbb{C}^{2n},Q_{n})$$
of rank $n$ as the subgroup of the general linear
group ${GL}(2n,\mathbb{C})$ preserving the skew symmetric bilinear form on
$\mathbb{C}^{2n}$ induced by
\begin{equation*}
\left[
\begin{array}{cc}
0 & Q_{n} \\
-Q_{n} & 0%
\end{array}%
\right] .
\end{equation*}%

Let $\{e_{a}\}$ be the elementary basis of $\mathbb{C}^{2n}$, and denote by $\{e_{a},e_{\overline{a}}\}$
the isotropic pairs, where $e_{\overline{a}%
}=e_{2n+1-a}$ for $1\leq a\leq n$. With respect to this bilinear form, we can take the subgroup $U_{n}$ of $G_n$ consisting of upper triangular matrices with $1$'s on the diagonal as a maximal unipotent subgroup of ${G}_{n}$. We let $U_{n}^{-}$ denote the subgroup of lower triangular matrices with $1$'s on the diagonal.

Let us identify ${G}_{n-1}$ with the subgroup of ${G}_{n}$ which acts as identity on the isotropic space spanned by $\{ e_{n}, e_{\overline{n}}$ \}. Then ${G}_{n-1}$ can be embedded in ${G}_{n}$ via%
\begin{equation}
\left[
\begin{array}{cc}
A & B \\
C & D%
\end{array}%
\right] \mapsto \left[
\begin{array}{ccc}
A & 0 & B \\
0 & I & 0 \\
C & 0 & D%
\end{array}%
\right]   \label{embedding}
\end{equation}%
where $A,B,C,D$ are $(n-1)\times (n-1)$ matrices, $I$ is the $2\times 2$
identity matrix, and $0$'s are the zero matrices of proper sizes.

\medskip

A \textit{Young diagram} is a finite left-justified array of boxes with weakly decreasing row lengths, such as
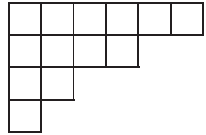
\begin{equation*}
\young(\ \ \ \ \ \ ,\ \ \ \ ,\ \ ,\ )
\end{equation*}%
We shall identify a Young diagram $F$ with its sequence of row lengths $(f_1,f_2,\dots)$. By reading column lengths of $F$, we obtain its associated Young diagram $F^{t}$ called the \textit{transpose} of $F$. We write $\ell (F)$ for the number of non-zero entries in $F$ and call it the \textit{length} of $F$. The Young diagram in the above example is $(6,4,2,1)$ or equivalently $(6,4,2,1,0,\dots)$ and $\ell (F)=4$. Its transpose $F^{t}$ is $(4,3,2,2,1,1)$.

Recall that every finite-dimensional irreducible representation of ${G}_{n}$ can be
uniquely labeled by a Young diagram with less than or equal to $n$ rows by
identifying its highest weight $\mu _{F}$ with Young diagram $%
F=(f_{1},\cdots ,f_{n})$:%
\begin{equation*}
\mu _{F}(t)=t_{1}^{f_{1}}\cdots t_{n}^{f_{n}}.
\end{equation*}%
Here $t$ is an element of the maximal torus ${T}_{n}$ of ${G}_{n}$%
\begin{equation*}
{T}_{n}=\left\{ diag(t_{1},\cdots ,t_{n},t_{n}^{-1},\cdots
,t_{1}^{-1})\right\}
\end{equation*}%
and $F\in \mathbb{Z}^{n}$ with $f_{1}\geq \cdots \geq f_{n}\geq 0$. See \cite[Section 3.2.1]{GW09} for details. We let $\tau _{n}^{F}$ denote the irreducible
representation of ${G}_{n}$ labeled by Young diagram $F$.

\medskip

To encode the branching multiplicities of $\tau _{n-1}^{D}$ in $\tau _{n}^{F}
$ for all pairs of Young diagrams $(D,F)$, we will use the following
semigroup%
\begin{equation*}
\Lambda _{n-1,n}=\{(D,F) \in \mathbb{Z}_{\geq 0}^{n-1} \times \mathbb{Z}_{\geq 0}^n :\ell (D)\leq n-1,\ell (F)\leq n\}
\end{equation*}%
and construct an algebra graded by $\Lambda _{n-1,n}$. The semigroup structure of $\Lambda _{n-1,n}$ is induced by the natural embedding of $ \Lambda _{n-1,n}$ in $ \mathbb{Z}_{\geq 0}^{n-1} \times \mathbb{Z}_{\geq 0}^n$ with addition.

\medskip

\subsection{Branching Algebra $\B$}

On the ring $\mathcal{R}({G}_{n})$ of regular functions over ${G}_{n}$, we
have the natural action of ${G}_{n}\times {G}_{n}$ given by%
\begin{equation}
\left( (g_{1},g_{2})\cdot f\right) (x)=f(g_{1}^{-1}xg_{2})
\label{double action}
\end{equation}%
for $f\in \mathcal{R}({G}_{n})$ and $\left( g_{1},g_{2}\right) \in {G}%
_{n}\times {G}_{n}$. With respect to this action, let us consider the
affine quotient of ${G}_{n}$ by ${U}_{n}^{-}\times 1$.

\begin{lemma}[{\protect\cite[Theorem 12.1.5]{GW09}}]
As a $G_{n}$-module under right translation the $\left( {U}_{n}^{-}\times 1\right) $-invariant subalgebra of\ $\mathcal{R%
}({G}_{n})$ contains every irreducible rational representation of ${G}_{n}$ with multiplicity one:
\begin{equation}
\mathcal{R}({G}_{n})^{{U}_{n}^{-}\times 1}=\bigoplus_{F\in \Lambda
_{n}}\tau _{n}^{F}  \label{flag algebra}
\end{equation}%
The algebra $\mathcal{R}({G}_{n})^{{U}_{n}^{-}\times 1}$ is graded by the semigroup of dominant weights for ${G}_{n}$ or
equivalently the set $\Lambda _{n}$ of Young diagrams of length less than or
equal to $n$.
\end{lemma}

In this setting, the irreducible representation $\tau _{n}^{F}$ is the
weight space of $\mathcal{R}({G}_{n})^{{U}_{n}^{-} \times 1}$ under the left action of
the maximal torus ${T}_{n}$ with weight $\mu _{(-F)}$, i.e.,%
\begin{equation*}
t\cdot f=\left( t_{1}^{-f_{1}}\cdots t_{n}^{-f_{n}}\right) f
\end{equation*}%
for $f\in \tau _{n}^{F}$ and $t\in T_{n}$. See \cite[Section 12.1.3]{GW09} for
further details.

\medskip

By highest weight theory (e.g., \cite[Section 3.2.1]{GW09}) the subspace of $\tau _{n}^{F}$ invariant under the
maximal unipotent subgroup of ${G}_{n-1}$ is spanned by highest weight
vectors of ${G}_{n-1}$-irreducible representations in $\tau _{n}^{F}$.
Therefore, the ${U}_{n-1}$-invariant subalgebra of $\mathcal{R}({G}_{n})^{{U%
}_{n}^{-}\times 1}$ contains the information of the branching multiplicities
for $({G}_{n},{G}_{n-1})$. That is,

\begin{proposition}
\label{branchingalgprop}
The $\left( {U}_{n}^{-}\times {U}_{n-1}\right) $-invariant subalgebra of the
ring $\mathcal{R}({G}_{n})$ decomposes as
\begin{eqnarray*}
&&\mathcal{R}({G}_{n})^{{U}_{n}^{-}\times {U}_{n-1}}
=\bigoplus_{(D,F) \in \Lambda _{n-1,n}}Hom_{{G}%
_{n-1}}(\tau _{n-1}^{D},\tau _{n}^{F})\otimes \left( \tau _{n-1}^{D}\right)
^{{U}_{n-1}}
\end{eqnarray*}
\end{proposition}

Note that by Schur's lemma the dimension of $Hom_{{G}_{n-1}}(\tau
_{n-1}^{D},\tau _{n}^{F})$ is equal to the multiplicity of $\tau _{n-1}^{D}$
in $\tau _{n}^{F}$. 
\begin{definition} 
We call $\mathcal{R}({G}_{n})^{{U}%
_{n}^{-}\times {U}_{n-1}}$ the \textit{branching algebra} for $({G}_{n},{G}%
_{n-1})$ and denote it by%
\begin{equation*}
\B=\mathcal{R}({G}_{n})^{{U}_{n}^{-}\times {U}_{n-1}}
\end{equation*}
\end{definition}
The algebra $\B$ has an action of $T_{n} \times T_{n-1}$, and the
weight space corresponding to the $T_{n}$ weight $(-F)$ and $T_{n-1}$ weight
$D$ is precisely the component 
$$
\B(D,F):=Hom_{{G}_{n-1}}(\tau _{n-1}^{D},\tau _{n}^{F})
\otimes \left( \tau _{n-1}^{D}\right)^{U_{n-1}}
$$ 
appearing in Proposition \ref{branchingalgprop}. Since the dimension of $(\tau _{n-1}^{D})^{U_{n-1}}$ is one, we can consider $\B(D,F)$ as the branching multiplicity space for $(G_n,G_{n-1})$:%
\begin{equation}\label{B-DF-mult}
\B(D,F) \cong Hom_{{G}_{n-1}}(\tau _{n-1}^{D},\tau _{n}^{F})
\end{equation}
Therefore the dimension of $\B(D,F)$ is exactly the multiplicity of the irreducible representation $\tau _{n-1}^{D}$ appearing in $\tau _{n}^{F}$.
Moreover, by keeping track of  $T_{n} \times T_{n-1}$ weights, it is straightforward to check that this defines a $\Lambda_{n-1,n}$-graded algebra structure on $\B$:
$$
\B = \bigoplus_{(D,F) \in \Lambda _{n-1,n}}\B(D,F).
$$

The dimensions of the graded components of $\B$ are given by the following combinatorial rule.
For two Young diagrams $F=(f_{1},f_{2},\cdots )$ and $D=(d_{1},d_{2},\cdots
) $, we say $D$ \textit{interlaces} $F$ and write $D\sqsubseteq F$, if 
$f_{i}\geq d_{i}\geq f_{i+1}$ for all $i$.

\begin{lemma}\label{G_branching}
\begin{enumerate}
\item (See, e.g., \cite[Theorem 8.1.5]{GW09})
For Young diagrams $D$ and $F$ with $(D,F)\in \Lambda
_{n-1,n}$, the multiplicity of $\tau _{n-1}^{D}$ in $\tau _{n}^{F}$ as a ${%
G}_{n-1}$ representation is nonzero if and only if
$$f_j \geq d_j \geq f_{j+2}$$
for $j=1,2,\cdots,n-1$. Here we assume $f_{n+1}=0$.

\item (See, e.g., \cite[Proposition 10.2]{Pr94})
The multiplicity of $\tau _{n-1}^{D}$ in $\tau _{n}^{F}$ as a ${G}_{n-1}$ representation is
equal to the number of Young diagrams $E=(e_1,\dots,e_n)$
satisfying the conditions $D\sqsubseteq E$ and $E\sqsubseteq F$, i.e., 
\begin{eqnarray*}
e_1 \geq d_1 \geq e_2 \geq \cdots \geq e_{n-1} \geq d_{n-1} \geq e_n ; \\
f_1 \geq e_1 \geq f_2 \geq e_2 \geq \cdots \geq e_{n-1} \geq f_{n} \geq e_n.
\end{eqnarray*}
\end{enumerate}
\end{lemma}

If $D\sqsubseteq E$ and $E\sqsubseteq F$ for some $E$ then, we say that the pair $(D,F)$  (or the triple $(D,E,F)$, if $E$ should be specified) satisfies the \textit {doubly interlacing condition}. Note that the branching for $(G_n, G_{n-1})$ is not multiplicity free, and  $D\sqsubseteq E\sqsubseteq F$ does not imply $D\sqsubseteq F$. We also note that the conditions in the second statement can be visualized as, by using the convention of Gelfand-Tsetlin patterns, 
\begin{equation*}
\begin{array}{ccccccccccc}
f_{1} &  & f_{2} &  & \cdots &  & f_{n} &  & f_{n} &  &  \\
& e_{1} &  & e_{2} &  & \cdots &  & e_{n-1} &  & e_{n} &  \\
&  & d_{1} &  & \cdots &  & d_{n-1} &   & d_{n} & &
\end{array}
\end{equation*}
where the entries are weakly decreasing from left to right along the diagonals.

\section{A standard monomial theory for $\B$}
In this section we show that $\B$ carries a standard monomial theory, in the sense that it has a natural basis
which satisfies a straightening algorithm. For the concept of standard monomial theory and its development, we refer to \cite{LR08} and  \cite{La03, Mu03}.

\subsection{Distributive Lattice for $\B$} \label{sectionASL}

Let ${M}_{2n}={M}_{2n}(\mathbb{C})$ be the space of $2n\times 2n$ complex
matrices. For a subset $C$ of $\{1,2,\cdots ,2n\}$ of cardinality $r$, let $$\delta _{C}:M_{2n} \rightarrow \mathbb{C}$$ denote
the map assigning a matrix $X\in {M}%
_{2n}$ the determinant of the $r\times r$ minor formed by taking rows $%
\{1,2,\cdots ,r\}$ and columns $\{c_{1},c_{2},\cdots ,c_{r}\}$:

\begin{equation}\label{del-det}
\delta _{C}(X)=\det
\begin{bmatrix}
x_{1,c_{1}} & x_{1,c_{2}} & \cdots & x_{1,c_{r}} \\
x_{2,c_{1}} & x_{2,c_{2}} & \cdots & x_{2,c_{r}} \\
\vdots & \vdots & \ddots & \vdots \\
x_{r,c_{1}} & x_{r,c_{2}} & \cdots & x_{r,c_{r}}%
\end{bmatrix}%
\end{equation}%
for $c_{1}<c_{2}<\cdots <c_{r}$. We note that $\delta _{C}$ is a weight
vector under the left and right multiplication of the diagonal subgroup of ${%
GL}_{2n}(\mathbb{C})$, i.e.,%
\begin{equation}
(t,s)\cdot \delta _{C}=\left( t_{1}^{-1}\cdots t_{r}^{-1}\right) \left(
s_{c_{1}}\cdots s_{c_{r}}\right) \delta _{C}  \label{joint weight vector}
\end{equation}%
In particular, the weight under the left and right actions encode the size
of $C$ and the entries of $C$ respectively.

\medskip

For the branching algebra $\B$, we shall use the following
subsets of $\{1,2,\cdots ,n,n+1\}$ for column indexing sets $C$:%
\begin{eqnarray*}
I_{i} &=&\{1,2,\cdots ,i\} \\
J_{j} &=&\{1,2,\cdots ,j,n\} \\
J_{j}^{\prime } &=&\{1,2,\cdots ,j,n+1\} \\
K_{k} &=&\{1,2,\cdots ,k,n,n+1\}
\end{eqnarray*}%
for $1\leq i\leq n-1,$ $0\leq j\leq n-1$\ and $0\leq k\leq n-2$, with the
convention of $J_{0}=\{n\}$, $J_{0}^{\prime }=\{n+1\}$, and $K_{0}=\{n,n+1\}$.

\begin{definition}\label{dlattice-def}
The \textbf{distributive lattice }$\mathcal{L}$\textbf{\ for }$\left(
{G}_{n},{G}_{n-1}\right) $ is the poset consisting of
\begin{equation*}
\left\{ I_{i},J_{j},J_{j}^{\prime },K_{k}:1\leq i\leq n-1,0\leq j\leq
n-1,0\leq k\leq n-2\right\}
\end{equation*}%
with the following partial order $\preceq $:%
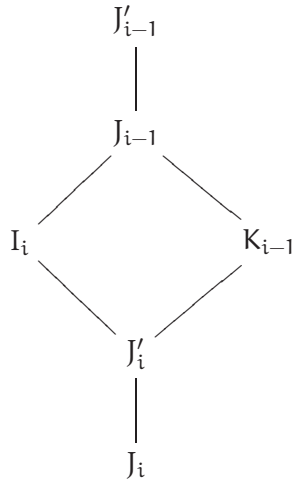
\begin{equation*}
\xymatrix{ & J_{i-1}^{\prime } \ar@{-}[d] & \\ & J_{i-1} & \\ I_{i}
\ar@{-}[ur]& & K_{i-1} \ar@{-}[ul]\\ & J_{i}^{\prime } \ar@{-}[ul]
\ar@{-}[d] \ar@{-}[ur]& \\ & J_{i} & \\ }
\end{equation*}%
for $1\leq i\leq n-1$.
\end{definition}

Note that the join and meet of incomparable elements can be easily found as%
\begin{eqnarray}
I_{i}\wedge K_{i-1} &=&J_{i}^{\prime }  \label{join-meet} \\
I_{i}\vee K_{i-1} &=&J_{i-1}  \notag
\end{eqnarray}%
for each $i$. This poset structure is very useful to organize the relations
among $\delta _{C}$ for $C\in \mathcal{L}$. The following can be
shown by a simple computation, or see \cite[Lemma 7.2.3]{BH98}.

\begin{proposition}\label{relations}
For $1\leq i\leq n-1$, the following identities hold%
\begin{equation*}
\delta _{I_{i}}\delta _{K_{i-1}}=\delta _{J_{i}^{\prime }}\delta
_{J_{i-1}}-\delta _{J_{i}}\delta _{J_{i-1}^{\prime }}
\end{equation*}%
over the space ${M}_{2n}$ and therefore over ${G}_{n}$.
\end{proposition}

Note that in the above Proposition, on the left hand side, $I_{i}$ and $K_{i-1}$ are
incomparable, while on the right hand side $J_{i}^{\prime }\preceq J_{i-1}$
and $J_{i}\preceq J_{i-1}^{\prime }$. In other words, by applying these
relations, we can express any quadratic monomial in $\{\delta _{C}:C\in \mathcal{L}\}$
 as a linear combination of quadratic monomials whose indices are
linearly ordered with respect to $\preceq $. In this sense, we call these
relations \textit{straightening relations}, and we can study the branching
algebra $\B$ in the context of an algebra with straightening
law (ASL) (cf. \cite{BH98, DEP82, Ei80}).

\begin{definition}[{\protect\cite{Ei80}}]\label{Ei-ASL}
Let $R$ be a ring, $A$ an $R$-algebra, $H$ a finite partially ordered set contained in $A$
which generates $A$ as an $R$-algebra. Then $A$ is an algebra with straightening law on $H$ over $R$, 
if
\begin{enumerate}
\item The algebra $A$ is a free $R$-module whose basis is the set of monomials of the form $\alpha_1 \cdots \alpha_k$
where $\alpha_1 \leq \dots \leq \alpha_{k}$ in $H$.

\item If $\alpha$ and $\beta$ in $H$ are incomparable, then 
$$ \alpha \beta = \sum_i c_i \gamma^{(i)}_1 \dots \gamma^{(i)}_{k_{i}}$$

where $\gamma_1^{(i)}\leq \cdots \leq \gamma_{k_i}^{(i)}$,
and,  for $i$ such that $c_i\neq 0$,  $r_1^{(i)} \leq \alpha,\beta$.
%
%
\end{enumerate}
\end{definition}
\medskip

In the following section we show the branching algebra $\B$ is an ASL on $\{\delta_C: C\in \mathcal{L}\}$ over $\mathbb{C}$.

\subsection{Standard monomials for $\B$}

Let us recall that a \textit{Young tableau} is a filling of a Young diagram with positive integers. A Young tableau is called a \textit{semistandard Young tableau}, if its entries in each row are weakly increasing from left to right, and its entries in each column are strictly increasing from top to bottom.

Young tableaux with entries from $\{1,\cdots,m \}$ may be identified with a product of determinants of minors over the space $M_{m}$ as follows. If the $i$-th column of a Young tableau $T$ contains the entries 
$${t_{1,i} < t_{2,i}< \dots <t_{r_i,i}}$$ 
for $1 \leq i \leq s$, then the corresponding polynomial in $\mathbb{C}[M_{m}]$ is
\begin{equation}\label{ssyt-det_bij}
 \delta_{T}=\prod_{1\leq i \leq s} \delta_{\{ t_{1,i},\dots,t_{r_i,i}\}}
\end{equation}%
where $\delta_{\{ t_{1,i},\dots,t_{r_i,i}\}}$ is as defined in (\ref{del-det}). See Example \ref{example} below. This type of correspondence between the set of Young tableaux and the set of products of determinants plays an important role in standard monomial theory of Grassmann and flag varieties. For this direction, we refer to, e.g., \cite{Km08, LR08, MS05}.

\medskip

Coming back to our setting, we will be considering monomials of the form $
\prod_{i=1}^s \delta_{C_i}
$
where $C_i \in \mathcal{L}$ (cf. Definition \ref{dlattice-def}).  These are considered as regular functions on $G_n$.  From (\ref{joint weight vector}), it is straightforward to see that every such monomial is a weight vector under the left and right actions of the maximal tori of ${G}_{n}$ and ${G}_{n-1}$ respectively. Moreover, by definition of the $U_n^- \times U_{n-1}$ action on $\mathcal{R}(G_n)$, the functions $\delta_C$, and hence their products $\prod_{i=1}^s \delta_{C_i}$, are invariant under $U_n^- \times U_{n-1}$.  In other words, $\prod_{i=1}^s \delta_{C_i} \in \B$. Let $\B' \subset \B$ be the subalgebra generated by $\delta_C$ for $C \in \mathcal{L}$.  Clearly, $\B'$ is spanned by monomials $\prod_{i=1}^s \delta_{C_i}$.

Now, we can form a Young tableau $T$ by concatenating finitely many elements $C_1, \dots, C_s$ chosen from $ \mathcal{L}$ allowing repetition. We further assume that the size of $C_i$ is not smaller than that of $C_{i+1}$ for all $i$. We note that the weakly increasing condition on the elements along the rows of $T$ can be replaced by the chain condition on  $ \mathcal{L}$ with respect to the partial order $\preceq$.  In other words, the elements ${C_1, \dots, C_s}$ of  $\mathcal{L}$ are columns of a semistandard tableau $T$ if and only if they are linearly ordered with respect to $\preceq$. This is true in a more general setting (cf. \cite[Remark 3.3]{Km08}). With this observation, we define standard monomials for $(G_n,G_{n-1})$ as follows.

\begin{definition}
A monomial $\prod \delta _{C_{i}}$  in $\{\delta _{C}:C\in \mathcal{L}\}$ is called a \textit{standard
monomial}  for $(G_n,G_{n-1})$ if the column indices $C_{i}$ form a multiple chain $\Delta $%
\begin{equation*}
\Delta =\left( C_{1}\preceq \cdots \preceq C_{r}\right)
\end{equation*}%
in the poset $\mathcal{L}$. We write $\delta _{\Delta }$ for $\prod
\delta _{C_{i}}$.
\end{definition}

The observation right after Proposition \ref{relations} now can be generalized in terms of standard monomials.

\begin{proposition}\label{spanning}
The set of standard monomials for $(G_n,G_{n-1})$ spans the subalgebra $\B'$ of $\B$ generated by 
$\{\delta_C : C \in \mathcal{L}\}$.
\end{proposition}
\begin{proof}
We want to show that every monomial can be expressed as a linear combination of standard monomials.  Observe that any monomial $\delta=\prod \delta _{C_{i}}$ can be expressed as
$$
\delta=(\delta_{I_1}\delta_{K_0})^{a_1} \cdots (\delta_{I_{n-1}}\delta_{K_{n-2}})^{a_{n-1}}\delta_{\Delta}
$$
where $\delta_{\Delta}$ is not divisible by $I_iK_{i-1}$ for $i=1,...,n-1$.  In particular, $\delta_{\Delta}$ is a standard monomial.  We prove the claim by induction on $a=\sum a_i$.

If $a=0$ then $\delta=\delta_{\Delta}$ is standard, and there is nothing to show.  Suppose $a>0$, and hence some $a_i>0$.  Then by Proposition \ref{relations},
$$
\delta=(\delta_{I_1}\delta_{K_0})^{a_1} \cdots(\delta_{I_i}\delta_{K_{i-1}})^{a_i-1} \cdots  (\delta_{I_{n-1}}\delta_{K_{n-2}})^{a_{n-1}}(\delta _{J_{i}^{\prime }}\delta
_{J_{i-1}}-\delta _{J_{i}}\delta _{J_{i-1}^{\prime }})\delta_{\Delta}
$$ 
Since $(\delta _{J_{i}^{\prime }}\delta
_{J_{i-1}}-\delta _{J_{i}}\delta _{J_{i-1}^{\prime }})\delta_{\Delta}$ is a linear combination of two standard monomials, each of which is not divisible by $I_iK_{i-1}$ for $i=1,...,n-1$, the result follows by induction.
\end{proof}
\begin{remark}\label{rem-homogeneous}
%
For a quadratic monomial $\delta _{C}\delta _{C'}$, we can simply apply Proposition \ref{relations} to express it as a linear combination of quadratic standard monomials%
\begin{equation}\label{quad-rel}
\delta _{C}\delta _{C'}=\sum_{h} {r_h} \delta _{D_h}\delta _{D'_h}
\end{equation}
Note that for all $h$, the numbers of entries equal to $i$ in the disjoint union $C \dot{\cup} C'$ and in the disjoint union $D_h \dot{\cup}D'_h$ are equal for $1 \leq i \leq n+1$.  Therefore, as tableaux,  $(D_h, D'_h)$ is obtained from $(C,C')$ just by rearranging the entries of $C $ and $C'$. In fact the only difference between the tableaux $(D_h, D'_h)$  is the position
of the entries $n$ and $n+1$.%

In general, once we have a linear combination of standard monomials for $\prod_{i} \delta_{C_i}$%
$$\prod_{i} \delta_{C_i}=\sum_{k} {s_k}(\prod_{i} \delta_{H_{k,i}})$$
then for each $k$, we have the semistandard tableau $H_k$ formed by $H_{k,i}$'s. Because of the reason explained above, for all $k$, the Young diagrams of the $H_k$'s are the same, and as tableaux, their only difference is the position of the entries $n$ and $n+1$.

\end{remark}

\begin{definition}
The \textit{shape} of a standard monomial $\delta _{\Delta }=\prod \delta
_{C_{i}}$ is $F/D$ with%
\begin{eqnarray*}
F &=&(f_{1},\cdots ,f_{n}) \\
D &=&(d_{1},\cdots ,d_{n-1})
\end{eqnarray*}%
where $F$ is the transpose of the Young diagram $(|C_{1}|,\cdots ,|C_{r}|)$
and $d_{k}$ in $D$ is the number of $k$'s in the disjoint union $\dot{\cup}
_{i=1}^{r}C_{i}$ for $1\leq k\leq n-1$. We write $sh(\delta _{\Delta })=F/D$.
\end{definition}

The following lemma is an immediate application of 
(\ref{joint weight vector}):

\begin{lemma}
\label{joint weight vector 2}Standard monomials $\delta _{\Delta }$ of shape
$F/D$ are weight vectors under the action of ${T}_{n}\times {T}_{n-1}$, i.e.,%
\begin{equation*}
(t,s)\cdot \delta _{\Delta }=\left( t_{1}^{-f_{1}}\cdots
t_{n}^{-f_{n}}\right) \left( s_{1}^{d_{1}}\cdots s_{n-1}^{d_{n-1}}\right)
\delta _{\Delta }
\end{equation*}%
where $t=diag(t_{1},\cdots ,t_{n},t_{n}^{-1},\cdots ,t_{1}^{-1})$ and $%
s=diag(s_{1},\cdots ,s_{n-1},1,1,s_{n-1}^{-1},\cdots ,s_{1}^{-1})$.  In particular,
$\delta_{\Delta} \in \B(D,F)$.
\end{lemma}

\medskip

Now let us count the number of standard monomials in $\B(D,F)$.

\begin{proposition}\label{counting}
There is a bijection
$$
\{\delta_{\Delta}:sh(\delta_{\Delta})=F/D \} \leftrightarrow \{E: D\sqsubseteq E\sqsubseteq F \}
$$
In particular, $dim \B(D,F)=\#\{\delta_{\Delta}:sh(\delta_{\Delta})=F/D \}$.
\end{proposition}
\begin{proof} 
The bijection is a variation on the conversion procedure between semistandard Young tableaux and Gelfand-Tsetlin patterns (cf. \cite[Section 8.1.2]{GW09}). Let us be more specific about the procedure in our case. 

By the definition of $sh(\delta_{\Delta})$, if we erase all the boxes with $n$ and $n+1$ in the semistandard tableau $\Delta$, then the remaining tableau gives the Young diagram $D$. As an intermediate step, if we erase only the boxes with $(n+1)$, then it gives a Young diagram $E$ such that $D\sqsubseteq E$ and $E\sqsubseteq F$. 

Conversely, given $E$ such that $D\sqsubseteq E\sqsubseteq F$, define a semistandard Young tableau of shape $F$ as follows: label the boxes of $F/E$ by $n+1$, the boxes of $E/D$ by $n$, and the remaining empty boxes by their row coordinate.  The standard monomial corresponding to $E$ is then constructed from this semistandard tableau. 

The last statement follows by (\ref{B-DF-mult}) and Lemma \ref{G_branching} (2).
\end{proof}

\begin{example}
\label{example} In studying branching multiplicity spaces for $({G}_{4},{G}%
_{3})$, the following monomial $\delta _{\Delta }$
\begin{equation*}
\delta _{\Delta }=\delta _{\{1234\}}\delta _{\{1245\}}\delta
_{\{125\}}\delta _{\{14\}}\delta _{\{5\}}
\end{equation*}%
as a regular function on ${G}_{4}$ is a standard monomial. By concatenating
its column indices to make the semistandard Young tableau%
\begin{equation*}
\young(11115,2224,345,45)
\end{equation*}%
we see the shape of $\delta _{\Delta }$ is $F/D=(5,4,3,2)/(4,3,1)$. By erasing all the boxes with $5$'s, we obtain Young diagram $E=(4,4,2,1)$:%
\begin{equation*}
\young(\ \ \ \ ,\ \ \ \ ,\ \ ,\ )
\end{equation*}%
Note that the triple $(D,E,F)$ satisfies the doubly interlacing condition.
\begin{equation*}
\begin{array}{ccccccccc}
5 &  & 4&  &3 &  & 2 &  &   \\
& 4 &  & 4 &  &2 &  & 1&    \\
&  & 4&  & 3 &  & 1& &
\end{array}
\end{equation*}
\end{example}

\begin{theorem}
\label{StThm} 
The branching algebra $\B$ is an ASL on $\{\delta_C: C\in \mathcal{L}\}$ over $\mathbb{C}$.
In particular, standard monomials form a $\mathbb{C}$-basis of the algebra $%
\B$, and
$$
\B(D,F)=span\{\delta_{\Delta}: sh(\delta_{\Delta})=F/D\}.
$$
\end{theorem}

\begin{proof}
Let us check the first condition in Definition \ref{Ei-ASL}. Recall that $\B'$ is the subalgebra of $\B$ generated by $\{\delta_{C}:C \in \mathcal{L}\}$, and that by Proposition \ref{spanning}, $\B'$ is spanned by standard monomials. Now consider the space $\B' \cap \B(D,F)$. By Lemma \ref{joint weight vector 2}, this space contains all the standard monomials of shape $F/D$, and therefore, it is spanned by standard monomials of shape $F/D$. By Proposition \ref{counting}, the number of standard monomials of shape $F/D$ is equal to the dimension of the space $\B(D,F)$. This shows that, for all doubly interlacing pairs $(D,F) \in \wedge_{n-1,n}$, the standard monomials of shape $F/D$ are a basis of $\B(D,F)$. Since $\B=\oplus_{(D,F)} \B(D,F)$ and the dimension of $\B(D,F)$ is zero unless $(D,F)$ satisfies the doubly interlacing condition (Lemma \ref{G_branching}), standard monomials form a $\mathbb{C}$-basis of $\B$. With Proposition \ref{relations}, which shows that $\B$ satisfies the second condition in Definition \ref{Ei-ASL}, this shows that the branching algebra $\B$ for $(G_n, G_{n-1})$ is an ASL on $\{ \delta_{C}:C \in \mathcal{L} \}$ over $\mathbb{C}$.
\end{proof}

\begin{definition}
The basis $\{ \delta_{\Delta} \}$ from Theorem \ref{StThm} is called  
the \textbf{standard monomial basis} of $\B$.
\end{definition}

\medskip

\section{The standard monomial basis as a canonical weight basis}
In this section we give an interpretation of the standard monomial basis of the previous section as a canonical weight basis.
In section \ref{irract} we recall a theorem in \cite{Yac09} which shows that the natural $SL_{2}$ action on $\B$ can be canonically extended to an action of an $n$-fold product of $SL_{2}$'s in such a way that the multiplicity spaces $\B(D,F)$ are irreducible.  As a corollary of this theorem we obtain a canonical decomposition of $\B$ into one-dimensional spaces.  We then show in section \ref{stbiswtb} that these one-dimensional spaces are exactly the spans of standard monomial basis elements.

\subsection{An irreducible action on the multiplicity spaces \label{irract}}

The branching algebra $\B$ carries a natural algebraic representation of $SL_{2}$.
Indeed, there is a copy of $SL_{2}$ in $G_{n}$ that commutes with
$G_{n-1} \subset G_{n}$ (cf. (\ref{embedding})). This copy of $SL_{2}$ acts
on the branching multiplicity spaces $\mathcal{B}(D,F)$, i.e. on the graded
components of $\B$. This action is described as follows: an element $b \in \mathcal{B}(D,F)$ is a $G_{n-1}$ equivariant map from $\tau_{n-1}^D$ to $\tau_n^F$, and given $x \in SL_{2}$, $x.b$ is another such morphism defined by $(x.b)(v)=x.b(v)$ for any $v\in \tau_{n-1}^D$.  Therefore $\B$ is a graded $SL_{2} $-algebra. We refer to this action as the ``natural'' $SL_{2}$ action on $\B$.

The branching multiplicity spaces are not irreducible $SL_{2}$-modules.  Indeed, they are an $n$-fold tensor product of irreducible $SL_{2}$-modules (see Theorem \ref{ExtThm} below).  Nevertheless, the natural $SL_{2}$ action can be uniquely extended to an irreducible action of a product of $SL_{2}$'s.  In this section we explain how this is done.

Let
$L=SL_{2} \times \cdots \times SL_{2}$ be the $n$-fold product of $%
SL_{2} $.  We want to construct an irreducible action of $L$ on $\B(D,F)$,
in such a way that the diagonally embedded $SL_{2} \subset L$ recovers
the natural action.  Notice that in this formulation $L$ is not the product of $SL_{2}$'s that lives in $G_{n}$.  Indeed, the latter
product of $SL_{2}$'s does not act on the multiplicity spaces.  The existence of this $L$-action is more subtle, and can only be ``seen'' by
considering all multiplicity spaces together, i.e. by considering the branching algebra.


For two Young diagrams $F=(f_{1},f_{2},\cdots )$ and $D=(d_{1},d_{2},\cdots
) $, the inequalities in doubly interlacing condition for $(D,F)$ given in  Lemma \ref{G_branching}, i.e.
\begin{equation*}
f_{i}\geq d_{i}\geq f_{i+2}
\end{equation*}%
do not constrain the relation between $d_{i}$ and $f_{i+1}$. In other words, we can have either $d_{i} \geq f_{i+1}$%
, or $d_{i} \leq f_{i+1}$, or both. This motivates the following:

\begin{definition}
\label{ordertype} An \textbf{order type} $\sigma$ is a word in the alphabet $%
\{\geq,\leq\}$ of length $n-1$.
\end{definition}

Suppose $(D,F) \in \Lambda _{n-1,n}$ and $\sigma = (\sigma_{1} \cdots
\sigma_{n-1})$ is an order type. Then we say $(D,F)$ is of \textbf{order
type $\sigma$} if for $i=1,...,n-1$,
\begin{equation*}
\left\{
\begin{array}{rl}
\sigma_{i} = \text{``}\geq\text{''} \Longrightarrow d_{i} \geq f_{i+1} &  \\
\sigma_{i} = \text{``}\leq\text{''} \Longrightarrow d_{i} \leq f_{i+1} &
\end{array}
\right.
\end{equation*}
For example, consider the pair $(D,F)$, where $F = (3,2,1)$ and $D= (3,0)$.
Since $d_{1} \geq f_{2}$ and $d_{2} \leq f_{3}$, the pair $(D,F)$ is of
order type $\sigma = (\geq\leq)$.

It will also be useful to introduce the notion of a \textbf{generalized order type}: if $%
d_{i}=f_{i+1}$ then we place an ``$=$'' in the $i^{th}$ position to denote
that $(D,F)$ satisfies order types with both $\geq$ and $\leq$
in the $i^{th}$ position. For example, if $F$ is as above and $D=(2,0)$ then we say $%
(D,F)$ is of generalized order type $(=\leq)$, since in this case $(D,F)$
satisfies both types $(\geq\leq)$ and $(\leq\leq)$.

Let $\Sigma$ be the set of order types, and for each $\sigma \in \Sigma$ set
\begin{equation*}
\Lambda _{n-1,n}(\sigma) = \{(D,F) \in \Lambda _{n-1,n} : (D,F) \text{ is of
order type } \sigma \}.
\end{equation*}
\begin{lemma}
For $\sigma \in \Sigma$, $\Lambda _{n-1,n}(\sigma)$ is a sub-semigroup of $%
\Lambda _{n-1,n}$.
\end{lemma}

\begin{proof}
Suppose $(D,F),(D',F') \in \Lambda _{n-1,n}(\sigma)$, and suppose $\sigma_i=\text{``}\geq\text{''}$.  Then $d_i \geq f_{i+1}$ and $d_i' \geq f_{i+1}'$, which of course implies that $d_i+d_i' \geq f_{i+1}+f_{i+1}'$, and hence $(D+D',F+F')\in \Lambda _{n-1,n}(\sigma)$.  The argument for $\sigma_i=\text{``}\leq\text{''}$ is entirely analogous.  
\end{proof}

Since $\B$ is $\Lambda_{n-1,n}$-graded, in particular $$\B(D,F)\B(D',F') \subset \B(D+D',F+F')$$
for $(D,F),(D',F') \in \Lambda_{n-1,n}$.  Therefore, by the above lemma,
\begin{equation*}
\B(\sigma) = \bigoplus_{(D,F) \in \Lambda _{n-1,n}(\sigma)}
\B(D,F)
\end{equation*}
is a subalgebra of $\B$.  Note that $\B(\sigma)$ has unit the trivial function on $G_n$, which is an element of the $(0,0)$-component. 

To each $(D,F) \in \Lambda_{n-1,n}$ we associate an irreducible $L$-module
as follows. Let $V_{k}$ be the irreducible $SL_{2}$-module of dimension $k+1$%
. Set $D=(d_{1},...,d_{n-1})$ and $F=(f_{1},...,f_{n})$, and let $(x_{1}
\geq y_{1} \geq \cdots \geq x_{n} \geq y_{n})$ be the non-increasing
rearrangement of the elements $\{d_{1},...,d_{n-1},f_{1},...,f_{n}\}$.
Define $r_{i}(D,F)=x_{i}-y_{i}$ for $i=1,...,n$, and let $\A(D,F)$ be the
irreducible $L$-module
\begin{equation*}
\A(D,F)=\bigotimes_{i=1}^{n}V_{r_{i}(D,F)}.
\end{equation*}

\begin{theorem}[{\protect\cite[Theorem 3.5]{Yac09}}]
\label{ExtThm}There is a unique representation $(\Phi,\B)$ of $%
L$ satisfying the following two properties:

\begin{enumerate}
\item For all $(D,F) \in \Lambda_{n-1,n}$, $\B(D,F)$ is an
irreducible $L$-invariant subspace of $\B$. If $\mathcal{B}%
(D,F)$ is nonzero, then $\B(D,F)$ is isomorphic to $%
\A(D,F) $.

\item For all $\sigma \in \Sigma$, $L$ acts as algebra automorphisms on $%
\B(\sigma)$.
\end{enumerate}

Moreover, $Res_{SL_{2}}^{L}(\Phi)$ recovers the natural action of $SL_{2}$
on $\B$.
\end{theorem}

Let $T_{SL_{2}}$ be the torus of $SL_{2}$ consisting of diagonal matrices.
Let $T_{L}=T_{SL_{2}}\times \cdots \times T_{SL_{2}}$ be the diagonal torus
of $L$. The action of $T_{L}$ on $\B(D,F)$ decomposes it uniquely into weight spaces, which, by the above theorem, are one-dimensional.

\begin{remark}
The $T_{L}$ weight spaces of $\B(D,F)$ are also weight spaces for the natural $SL_{2}$-action,
via the diagonally embedding $T_{SL_{2}} \subset L$.  Moreover, $T_{L}$ is the unique maximal torus of $L$ containing $T_{SL_{2}}$.
Therefore, the decomposition we obtain in this way is the unique decomposition of $\mathcal{B}(D,F)$ into spaces which are
simultaneously weight spaces for a
torus of $L$ and weight spaces for $T_{SL_{2}}$.  Moreover, the choice of the torus $T_{SL_{2}}$ is induced by our choice of torus of $G_{n}$.
In other words, the decomposition of $\B(D,F)$ into one dimensional spaces depends only the choice of torus of $G_{n}$.
\end{remark}

We now make this decomposition precise.  Suppose $(D,F) \in \Lambda_{n-1,n}$
and $(D,F)$ satisfies the doubly interlacing condition (so that $\B(D,F)$ is nonzero).
Then the weight spaces of $T_{L}$ on $\A(D,F)$, and hence $\mathcal{B}%
(D,F)$, are indexed by Young diagrams $E$ satisfying the condition 
$D\sqsubseteq E \sqsubseteq F$. Indeed, the diagram $E=(e_{1},e_{2},...)$
corresponds to the weight
\begin{equation}  \label{EqWt}
(t_{1},...,t_{n}) \in T_{L} \mapsto \prod_{i=1}^{n}
t_{i}^{2e_{i}-x_{i}-y_{i}}
\end{equation}
(cf. (Lemma 7.1, \cite{Yac09})). Let $\B(D,E,F)$ denote the
one dimensional weight space of $\B(D,F)$ parameterized by $E$.

\begin{corollary}
\label{ExtCor} There is a canonical decomposition of $\B$ into
one dimensional $T_{L}$ weight spaces
\begin{equation*}
\B=\bigoplus_{D\sqsubseteq E \sqsubseteq F}\mathcal{B}(D,E,F).
\end{equation*}%
In particular, $\B$ has a $T_{L}$ weight basis which is unique
up to scalar.
\end{corollary}
%
%
%

\subsection{The standard monomial basis is the canonical weight basis \label{stbiswtb}}

We now show that the standard monomials basis defined in Theorem \ref{StThm}
is compatible with the canonical decomposition appearing in Corollary \ref{ExtCor}.
In other words, the standard monomial basis is the unique (up to scalar) $T_{L}$ weight
basis of $\B$.

Suppose $\Delta \subset \mathcal{L}$ is a chain, and the
corresponding standard monomial $\delta_{\Delta}$ has shape $F/D$. We say
$\Delta$ is of order type $\sigma \in \Sigma$ if the pair $(D,F)$ is of type $%
\sigma$. The following lemma shows that order types are intimately connected
to the distributive lattice $\mathcal{L}$.

\begin{lemma}
\label{mainlemma} A set of column indices $\{C_{i}:i=1,...,r\}$ form a chain
$\Delta=(C_{1}\preceq \cdots \preceq C_{r})$ if, and only if, they satisfy a
common order type $\sigma \in \Sigma$.
\end{lemma}

\begin{proof}
First we note that
\begin{eqnarray*}
I_{i} &\text{is of generalized type}& (= \cdots \geq \cdots =) \\
J_{j} &\text{is of generalized type}& (= \cdots =) \\
J_{j}^{\prime } &\text{is of generalized type}& (= \cdots =) \\
K_{k} &\text{is of generalized type}& (= \cdots \leq \cdots =)
\end{eqnarray*}%
where in the first line the ``$\geq$" sign appears in the $i^{th}$ position,
and in the last line the ``$\leq$" sign appears in the $k+1^{th}$ position.

Now suppose a set of column indices $\{C_{i}:i=1,...,r\}$ form a chain in $%
\mathcal{L}$. Then for all $i \geq 1$ we know that
\begin{equation*}
\{I_{i},K_{i-1}\} \nsubseteq \{C_{1},...,C_{r}\}.
\end{equation*}
Define a generalized order type $\sigma = (\sigma_{1} \cdots \sigma_{n-1})$
by
\begin{equation*}
\sigma_{i} = \left\{
\begin{array}{lr}
\geq \text{ if } I_{i} \in \{C_{1},...,C_{r}\} &  \\
\leq \text{ if } K_{i-1} \in \{C_{1},...,C_{r}\} &  \\
= \text{ otherwise } &
\end{array}
\right.
\end{equation*}
Clearly the set $\{C_{1},...,C_{r}\}$ satisfies the type $\sigma$.

Conversely, suppose the elements of $\Delta = \{C_{1},...,C_{r}\}$
satisfy a common order type $\sigma=(\sigma_{1} \cdots \sigma_{n-1})$. Then
\begin{eqnarray*}
\sigma_{i} = \text{``}\geq\text{''} \Longrightarrow I_{i} \in \Delta,
K_{i-1} \not \in \Delta \\
\sigma_{i} = \text{``}\leq\text{''} \Longrightarrow I_{i} \not \in \Delta, K_{i-1} \in \Delta \\
\sigma_{i} = \text{``}=\text{''} \Longrightarrow I_{i} \not \in \Delta,
K_{i-1} \not \in \Delta
\end{eqnarray*}%
Therefore $\{I_{i},K_{i-1}\} \nsubseteq \Delta$ for all $i$, i.e.
$\Delta$ is a chain in $\mathcal{L}$.
\end{proof}

\begin{theorem}
\label{CanThm} The standard monomial basis is the unique (up to scalar) $T_{L}$ weight
basis of the representation $(\Phi,\B)$ of $L$.
\end{theorem}

\begin{proof}
First we prove that if $C \in \mathcal{L}$ then $\delta_{C}$ is a
weight vector for the action of $T_{L}$ on $\B$. Suppose $F/D$
is the shape of $\delta_{C}$. Since $D$ and $F$ are both columns, by Theorem %
\ref{ExtThm}(1) $B(D,F)$ is isomorphic either to a trivial $L$-module, or to
\begin{equation*}
V_{0} \otimes \cdots V_{1} \cdots \otimes V_{0}
\end{equation*}
where the term $V_{1}$ occurs, say, in the $i^{th}$ position.

In the first case, $\delta_{C}$ is clearly a weight vector since it is
invariant under $L$. Consider now the second case, and let $\vec{t}%
=(t_{1},...,t_{n}) \in T_{L}$. Then
\begin{eqnarray*}
\Phi(\vec{t})(\delta_{C}) &=& \Phi((t_{i},...,t_{i}))(\delta_{C}) \\
&=& t_{i}.\delta_{C}
\end{eqnarray*}
where the second equality follows since $Res_{SL_{2}}^{L}(\Phi)$ is the
natural action of $SL_{2}$ on $\B$ (by Theorem \ref{ExtThm}).
So now it suffices to see that $\delta_{C}$ is a weight vector under the
natural torus action of $T_{SL_{2}}$ on $\B$. Indeed, for $t
\in T_{SL_{2}}$ we have
\begin{eqnarray*}
t.\delta_{I_{i}} &=& \delta_{I_{i}} \\
t.\delta_{J_{j}} &=& t\delta_{J_{j}} \\
t.\delta_{J^{\prime}_{j}} &=& t^{-1}\delta_{J^{\prime}_{j}} \\
t.\delta_{K_{k}} &=& \delta_{K_{k}}
\end{eqnarray*}
where $1\leq i\leq n-1,$ $0\leq j\leq n-1$\ and $0\leq k\leq n-2$.

We've shown so far that for any $C \in \mathcal{L}$, $\delta_{C}$ is a
weight vector for the action of $T_{L}$ on $\B$. We now show
this is true for any standard monomial $\delta_{\Delta}=\delta_{C_{1}}
\cdots \delta_{C_{r}}$. Indeed, by Lemma \ref{mainlemma}, the column indices
$\{ C_{1},...,C_{r} \}$ satisfy a common order type. Then by Theorem \ref%
{ExtThm}(2), for any $l \in L$
\begin{equation*}
\Phi(l)(\delta_{\Delta}) = \Phi(l)(\delta_{C_{1}}) \cdots
\Phi(l)(\delta_{C_{r}}).
\end{equation*}
Since each $\delta_{C_{i}}$ is a weight vector for the action of $T_{L}$, it
follows that $\delta_{\Delta}$ is also. By Theorem \ref{StThm} and Corollary %
\ref{ExtCor} we conclude that the standard monomials are the unique (up to
scalar) weight basis of the representation $(\Phi,\B)$ of $L$.
\end{proof}

We remark that our labeling of standard monomials is compatible with the
decomposition of $\B$ appearing in Corollary \ref{ExtCor}.
Indeed, suppose $\delta_{\Delta}$ is a standard monomial. In section \ref%
{sectionASL} we showed how to associate a triple of doubly interlacing Young
diagrams $D_{\Delta}\sqsubseteq  E_{\Delta} \sqsubseteq F_{\Delta}$ to $%
\delta_{\Delta}$. Then we have:
\begin{equation}  \label{EqComp}
\delta_{\Delta} \in \B(D_{\Delta},E_{\Delta},F_{\Delta}).
\end{equation}

\begin{example}
Consider $\delta_{\Delta}$ as in Example \ref{example}. Let $\vec{t}%
=(t_{1},...,t_{4}) \in T_{L}$. By the same reasoning as in the proof of
Theorem \ref{CanThm} we compute that
\begin{equation*}
\Phi(\vec{t})(\delta_{\Delta})=t_{1}^{-1}t_{2}t_{3}^{-1}t_{4}\delta_{\Delta}.
\end{equation*}
On the other hand, using (\ref{EqWt}), it's easy to see that $\mathcal{B}%
(D_{\Delta},E_{\Delta},F_{\Delta})$ is a weight space of $T_{L}$
corresponding to the character $\vec{t} \mapsto
t_{1}^{-1}t_{2}t_{3}^{-1}t_{4}$.
\end{example}
%



\medskip

\section{Toric Degeneration and Hibi Algebra for $\left( {G}_{n},{G}%
_{n-1}\right) $}

In this section, we show that $\B$ can be flatly deformed into
an affine semigroup ring. This will provide another combinatorial
description of branching multiplicity spaces.

\subsection{Flat Deformation}

From Theorem \ref{StThm}, we can realize the
branching algebra $\B$\ as an ASL, i.e., the quotient algebra%
\begin{equation*}
\B\cong \mathbb{C}[z_{C}:C\in \mathcal{L}]/\mathcal{I}
\end{equation*}%
whose defining ideal $\mathcal{I}$ is generated by%
\begin{equation}
\left\{ z_{I_{i}}z_{K_{i-1}}-z_{J_{i}^{\prime
}}z_{J_{i-1}}+z_{J_{i}}z_{J_{i-1}^{\prime }}:1\leq i\leq n-1\right\}
\label{relations2}
\end{equation}

On the other hand, we can define a semigroup ring whose multiplicative
structure is compatible with the lattice structure of $\mathcal{L}$
(cf. \cite{Hi87}).

\begin{definition}
The \textbf{Hibi algebra} $\mathcal{H}$ over $\mathcal{L}$ is
the quotient of the polynomial ring $\mathbb{C}[z_{C}:C\in \mathcal{L}]$ by the ideal $\mathcal{I}_{0}$ generated by%
\begin{equation*}
\left\{ z_{I_{i}}z_{K_{i-1}}-z_{J_{i}^{\prime }}z_{J_{i-1}}:1\leq i\leq
n-1\right\} .
\end{equation*}
\end{definition}

We note that each generator of $\mathcal{I}_{0}$ can be written as $%
z_{I_{i}}z_{K_{i-1}}-z_{I_{i}\wedge K_{i-1}}z_{I_{i}\vee K_{i-1}}$ by (\ref%
{join-meet}), and also it is the first two terms of the longer relation (\ref%
{relations2}) for $\B$. In what follows, by using an analog of
the SAGBI degeneration (\cite[Theorem 1.2]{CHV96}), we show that $\mathcal{B}$
is a flat deformation of $\mathcal{H}$.

\begin{theorem}
\label{Flat Deformation}The branching algebra $\B$ can be
flatly deformed into the Hibi algebra $\mathcal{H}$ over $\mathcal{L}$.
\end{theorem}

\begin{proof}
Our goal is to construct a flat $\mathbb{C}[t]$ module $\mathcal{R}$ whose
general fiber is isomorphic to $\B$ and special fiber is
isomorphic to the Hibi algebra $\mathcal{H}$. Let us impose a
filtration on $\B$ by giving the following weight on each
monomial. Fix a large integer $N$ greater than $2n$, and for $%
C=\{c_{1}<\cdots <c_{a}\}\in \mathcal{L}$ we define its weight as%
\begin{equation*}
wt(C)=\sum_{r\geq 1}c_{r}N^{n-r}
\end{equation*}%
Also, we define the weight of $\delta _{C}$ as the weight $wt(C)$ of its
indexing set $C$, and the weight of a product $\prod \delta _{C_{i}}$
as the sum $\sum wt(C_{i})$ of the weights of its factors.

Recall that from Theorem \ref{StThm}, every element in $\B$
can be expressed uniquely as a linear combination of standard monomials $%
\delta _{\Delta }=\prod \delta _{C_{i}}$ with multiple chains $\Delta
=(C_{1},\cdots ,C_{k})$. Set $\mathsf{F}_{d}^{wt}(\B)$ to be
the space spanned by standard monomials whose weights are not less than $d$:%
\begin{equation*}
\left\{ \delta _{\Delta }:wt(\delta _{\Delta })\geq d\right\} .
\end{equation*}%
This filtration $\mathsf{F}^{wt}=\{\mathsf{F}_{d}^{wt}\}$ is well defined
from the following observation: every product $\prod \delta _{C_{i}}$ can be
expressed as a linear combination of standard monomials with bigger weights,
because in the straightening laws (\ref{relations2}) we have
\begin{eqnarray*}
wt(I_{i})+wt(K_{i-1}) &=&wt(J_{i}^{\prime })+wt(J_{i-1}) \\
&=&\sum_{1\leq r\leq i-1}rN^{n-r}+(n+i)N^{n-i}+(n+1)N^{n-i-1} \\
wt(J_{i})+wt(J_{i-1}^{\prime }) &=&\sum_{1\leq r\leq
i-1}rN^{n-r}+(n+i+1)N^{n-i}+nN^{n-i-1}
\end{eqnarray*}%
and therefore for each $i$,
\begin{eqnarray*}
wt(I_{i})+wt(K_{i-1}) &=&wt(J_{i}^{\prime })+wt(J_{i-1}) \\
&<&wt(J_{i})+wt(J_{i-1}^{\prime })
\end{eqnarray*}

Then we can construct the Rees algebra $\mathcal{R}$ of $\B$
with respect to $\mathsf{F}^{wt}$:%
\begin{equation*}
\mathcal{R}=\bigoplus_{d\geq 0}\mathsf{F}_{d}^{wt}(\B)t^{d}
\end{equation*}%
and by the general property of the Rees algebras, it is flat over $\mathbb{C}%
[t]$ with its general fiber isomorphic to $\B$ and special
fiber isomorphic to the associated graded ring.

For all incomparable pairs $\left( A,B\right) =\left( I_{i},K_{i-1}\right)
\in \mathcal{L}$, since $wt(A)+wt(B)=wt(A\wedge B)+wt(A\vee B)$, $%
\delta _{A}\delta _{B}$ and $\delta _{A\wedge B}\delta _{A\vee B}$ belong to
the same associated graded component of $\mathcal{R}$. Therefore, we have $%
y_{A}\cdot _{gr}y_{B}=y_{A\wedge B}\cdot _{gr}y_{A\vee B}$ where $y_{C}$ are
elements corresponding to $\delta _{C}$ in the associated graded ring of $%
\mathcal{R}$. Then it follows that the associated graded ring of $\mathcal{R}
$ is isomorphic to the Hibi algebra $\mathcal{H}$ over $\mathcal{L}$.
\end{proof}

\medskip

\subsection{Affine Semigroup}

Now we want to realize the Hibi algebra $\mathcal{H}$ over $\mathcal{L}$
 as an affine semigroup ring, i.e., a ring generated by a finitely
generated semigroup isomorphic to a subsemigroup of $\mathbb{Z}^{N}$
containing $0$ for some $N$ (cf. \cite[Section 6]{BH98}). Since $\B$ is a
flat deformation of $\mathcal{H}$, we expect that combinatorial
properties of branching rules give rise to the affine semigroup structure of
$\mathcal{H}$.

Let us define the poset $\Gamma $ consisting of $t_{j}^{(i)}$ for $n-1\leq
i\leq n+1$ and $1\leq j\leq \min (i,n)$ which we shall arrange as%
\begin{equation*}
\Gamma =\left\{
\begin{array}{ccccccccc}
t_{1}^{(n+1)} &  & t_{2}^{(n+1)} &  & \cdots &  & t_{n}^{(n+1)} &  &  \\
& t_{1}^{(n)} &  & t_{2}^{(n)} &  & \cdots &  & t_{n}^{(n)} &  \\
&  & t_{1}^{(n-1)} &  & \cdots &  & t_{n-1}^{(n-1)} &  &
\end{array}%
\right\}
\end{equation*}%
with $t_{j}^{(i+1)}\geq t_{j}^{(i)}\geq t_{j+1}^{(i+1)}$ for all $i$
and $j$.

The set $\mathcal{P}(\Gamma )$ of all order preserving maps from $\Gamma $
to non-negative integers forms a monoid generated by the characteristic
functions $\chi _{S}$ on $S=\Gamma /S'$ for order decreasing subsets $S'$ of $\Gamma$%
\begin{equation*}
\chi _{S}(t_{j}^{(i)})=\left\{
\begin{array}{c}
1\text{ if }t_{j}^{(i)}\in S \\
0\text{ if }t_{j}^{(i)}\notin S%
\end{array}%
\right.
\end{equation*}

Furthermore, by imposing the following partial order on the set of order
decreasing subsets of $\Gamma $, we can identify the Hibi algebra $\mathcal{H}$
with the semigroup ring $\mathbb{C}[\mathcal{P}]$\ of $\mathcal{P}=%
\mathcal{P}(\Gamma )$: for two order decreasing subsets $S'_1$ and $S'_2$
of $\Gamma $, we say $S'_1$ is bigger than $S'_2$, if $S'_2 \subseteq
S'_1$ as sets.

\begin{lemma}
There is an order isomorphism between $\mathcal{L}$ and the set of
order decreasing subsets of $\Gamma $.
\end{lemma}

This is an easy computation similar to \cite[Theorem 3.8]{Km08}. Let us
specify this isomorphism. For each $C\in \mathcal{L}$, we define the
complement $S_{C}$ of the corresponding order decreasing subset $S'_{C}$ of $\Gamma $ as the union of%
\begin{equation*}
S_{C}^{(k)}=\left\{ t_{1}^{(k)},t_{2}^{(k)},\cdots ,t_{m(k)}^{(k)}\right\}
\end{equation*}%
for $n-1\leq k\leq n+1$ where $m(k)$ is the number of entries in $C$ less
than or equal to $k$. It is straightforward to check that this
correspondence gives an order isomorphism. This, in fact, is an example of
Birkhoff's representation theorem or the fundamental theorem for finite
distributive lattices (\cite[Theorem 3.4.1]{Sta97}). See also the example below.

\begin{example}\label{char-func}
Let us consider the following elements from the distributive lattice $\mathcal{L}$ for $(G_4,G_3)$: $A=[1,2,4,5]$, $B=[1,2,5]$, and $C=[1,4]$. Then the corresponding character functions can be visualized as, by identifying them with their values at $t^{(b)}_a \in \Gamma$,
\begin{eqnarray*}
\chi_{S_A} &=& \left\{
\begin{array}{cccccccc}
1& &1& &1& &1&  \\
 &1& &1& &1& &0 \\
 & &1& &1& &0&
\end{array}%
\right\} \\
\chi_{S_B} &=& \left\{
\begin{array}{cccccccc}
1& &1& &1& &0&  \\
 &1& &1& &0& &0 \\
 & &1& &1& &0&
\end{array}%
\right\} \\
\chi_{S_C} &=& \left\{
\begin{array}{cccccccc}
1& &1& &0& &0&  \\
 &1& &1& &0& &0 \\
 & &1& &0& &0&
\end{array}%
\right\}
\end{eqnarray*}
Note that for $T=A,B,C$, the number of $1$'s in the first, second, and third row of $\chi_T$ is the number of entries less than or equal to $5$, $4$, and $3$ in $T$ respectively. The order $A \preceq B\preceq C$ can be related to the inclusion order on order decreasing subsets $S'_T = \chi^{-1}_{S_T} (0)$ of $\Gamma$:
$$\chi^{-1}_{S_A} (0) \subseteq  \chi^{-1}_{S_B} (0)\subseteq  \chi^{-1}_{S_C} (0).$$
\end{example}

\medskip

\begin{proposition}
There is a bijection between the set of multiple chains in $\mathcal{L}$
and the set of order preserving maps from $\Gamma $ to non-negative
integers.
\end{proposition}

\begin{proof}
The bijection in the above lemma provides the bijection between $\mathcal{L}$
and the set of characteristic functions on the complements of order increasing subsets
of $\Gamma $. This map can be extended to the multiple chains in $\mathcal{L}$
as follows. Let $\Delta =(C_{1}\preceq \cdots \preceq C_{c})$ be a
multiple chain in $\mathcal{L}$ and for each $i$, let $\chi _{i}$ be
the characteristic function on $S_{C_{i}}$
corresponding to $C_{i}$ given in the above lemma. Then we can consider
the following correspondence:%
\begin{equation}
\Delta =(C_{1}\preceq \cdots \preceq C_{c})\mapsto p(\Delta
)=\sum\limits_{r=1}^{c}\chi _{i}.  \label{adding procedure}
\end{equation}

Recall that starting from $\Delta $, by considering Young diagrams containing entries less than or equal to $n-1$,$n$, and $n+1$ we obtain Young diagrams $D$,$E$, and $F$ respectively with $D\sqsubseteq E \sqsubseteq F$. Now let us consider the following order preserving map $p:\Gamma \rightarrow \mathbb{Z}_{\geq 0}$%
\begin{eqnarray*}
p(t_{i}^{(n+1)}) &=&f_{i} \\
p(t_{i}^{(n)}) &=&e_{i} \\
p(t_{k}^{(n-1)}) &=&d_{k}
\end{eqnarray*}%
for $1\leq i\leq n,1\leq k\leq n-1$. Then from the construction of $\chi _{i}
$ corresponding to $C_{i}$, one can check that $p=p(\Delta )$ and
this correspondence gives a bijection. For further details, see \cite{Km08}%
\cite{How05}. See also the example below.
\end{proof}

\begin{example}
Let us consider the chain $A \preceq B \preceq C$ given by the elements in Example \ref{char-func}. By concatenating them we have the semistandard tableau $\Delta$ of the shape $F/D=(3,3,2,1)/(3,2)$:
\begin{equation*}
\young(111,224,45,5)
\end{equation*}
Then the corresponding order preserving map given in the above proof is 
$\chi_{S_A}+\chi_{S_B}+\chi_{S_C}$: 
\begin{equation*}
p(\Delta) = \left\{
\begin{array}{cccccccc}
3& &3& &2& &1&  \\
 &3& &3& &1& &0 \\
 & &3& &2& &0&
\end{array}
\right\}
\end{equation*}

Note that the first row $(3,3,2,1)$ corresponds to the Young diagram $F$ of the tableau $\Delta$. By erasing the boxes with $5$ in the tableau, we obtain the Young diagram corresponding to the second row $E=(3,3,1,0)$, and finally by erasing the boxes with $4$, we obtain the Young diagram $D$ corresponding to the third row $(3,2,0)$. 
\end{example}

In fact, from the multiplicative structure given in (\ref{adding procedure}),
this bijection is a semigroup isomorphism. Therefore, we have

\begin{corollary}
The Hibi algebra $\mathcal{H}$ over $\mathcal{L}$ is
isomorphic to the semigroup ring $\mathbb{C}[\mathcal{P}]$\ generated by $%
\mathcal{P}=\mathcal{P}(\Gamma )$.
\end{corollary}

Consequently, $Spec(\mathcal{H})$ is an affine toric variety in the
sense of \cite{St95, MS05}, and Theorem \ref{Flat Deformation} shows that $Spec(\B)$ is a toric deformation of $Spec(\mathcal{H})$.

\medskip\

\end{document}